\documentclass{fundam}

\publyear{22}
\papernumber{2102}
\volume{185}
\issue{1}

\usepackage{url} 
\usepackage[ruled,lined]{algorithm2e}
\usepackage{graphicx}
\usepackage{tikz}
\usetikzlibrary{automata, positioning, arrows}

\begin{document}

\title{On the Relationship Between Semiring-Induced Valuation Algebra Isomorphisms and Semiring Isomorphisms}

\address{School of Mathematics and Statistics,
	Jiangsu Normal University, 57 Heping Road, Xuzhou, China}

\author{Xuechong Guan\\
School of Mathematics and Statistics \\
Jiangsu Normal University \\ 57 Heping Road, Xuzhou, China\\
guanxc{@}foxmail.com
 } 

\maketitle

\runninghead{Xuechong Guan}{On the Relationship Between Semiring-Induced Valuation Algebra Isomorphisms and Semiring Isomorphisms}

\begin{abstract}
  This paper establishes a bidirectional relationship between isomorphisms of semirings and isomorphisms of their induced valuation algebras. We show that a semiring isomorphism is sufficient for a valuation algebra isomorphism. Conversely, for additively idempotent semirings without multiplicatively idempotent zero divisors, the existence of a valuation algebra isomorphism implies that the underlying semirings are isomorphic.
\end{abstract}


\section{Introduction}\label{sec:0}
A valuation algebra is an abstract formalism rooted in artificial intelligence that enables local computation and subsumes a wide range of established frameworks, including constraint systems, belief function systems, database theory, and formal logic \cite{RefB1,VBS,shafer,Wilson}. Its core structural properties have also been formally identified and validated in the theories of imprecise probabilities and soft sets \cite{probabilities,guan}. In recent years, significant new advances in valuation algebra research have emerged in domains including approximate reasoning, local computation, and distributed systems \cite{haenni,pouly,concurrent}.

A semiring, the second key algebraic structure examined in this work, was first explicitly introduced by H. S. Vandiver in 1934 and has since become a central subject of systematic mathematical investigation. Over subsequent decades, semiring theory has undergone continuous theoretical development, with broad applications established across diverse fields such as linear algebra, optimization theory, and theoretical computer science. 
Within the formal framework of valuation algebras, semiring-induced constraint systems were first systematically developed in prior work \cite{constraint1,Constraint}, and subsequent follow-up studies -- most notably the detailed analyses presented in \cite{examples} -- have contributed a rich body of illustrative examples and practical applications. References \cite{semiring,pouly} formally established the mechanism by which valuation algebras are induced by semirings, and rigorously characterized the deep structural relationship between the algebraic properties of valuation algebras and their underlying semirings. These developments have established semiring-induced valuation algebras as a significant research topic within the field.

The concept of valuation algebra homomorphism was first proposed in Reference \cite{information}. It explores the connections between valuation algebras from the perspective of algebraic structures, and conducts in-depth research on the fundamental problems related to homomorphisms. 
For this specific class of valuation algebras induced by semirings, it is a topic worthy of research to systematically investigate the mutual influence between the relationship among semirings and the relationship among valuation algebras.
The primary objective of the present work is to investigate the correspondence between the homomorphisms of semiring-induced valuation algebras and the homomorphisms of their underlying semirings, thereby elucidating the intrinsic algebraic structural connections between these two closely interrelated research domains. To achieve this goal, Section 2 of this paper presents the necessary preliminaries on semirings and valuation algebras, together with a systematic introduction to the notations adopted throughout the entire manuscript. Building on these preliminaries, Section 3 and Section 4 present a sequence of proofs that characterize the intrinsic, bidirectional relationship between semiring homomorphisms and the homomorphisms of their induced valuation algebras.

\section{Notations and Preliminaries}\label{sec:1}

In this paper, variables will be designated by capital letters (e.g., $X,Y,\cdots$). Lower-case letters such as $x,y,\cdots$ denote sets of variables. The symbol $\Omega_X$ represents the finite set of possible values of $X$, which is referred to as the {\it frame} of $X$. For a nonempty set of variables $s$, let $\Omega_s$ denote the Cartesian product of the frames $\Omega_X$ for all $X\in s$, i.e., $$\Omega_s= \prod_{X\in s} \Omega_X.$$ This set $\Omega_s$ is called the frame of $s$. If $s$ is empty, we define $\Omega_\emptyset = \{\epsilon\}$. An element of $\Omega_s$ is called a configuration of $s$, and can be expressed in the form $(\omega_X)_{X \in s}$, where $\omega_X\in \Omega_X$. For convenience, we also use bold-faced lower-case letters (e.g., ${\bf x},{\bf y},\cdots$) to denote these configurations. 

If ${\bf x}$ is a configuration with domain $s$ and $t\subseteq s$, then ${\bf x}^{\downarrow t}$ denotes the projection of ${\bf x}$ onto the subdomain $t$. In particular, ${\bf x}^{\downarrow \emptyset} = \epsilon$. A configuration ${\bf x}\in \Omega_s$ can be decomposed into components corresponding to finite disjoint subsets of $s$, such as ${\bf x} = ({\bf x}^{\downarrow t },  {\bf x}^{\downarrow s-t})$, where $t\subseteq s$.

The following introduces terminology related to semirings. Let $A$ be a set with two binary operations $+$ and $\times$ defined on it. We call a tuple $\langle A,+,\times\rangle $ a semiring, if both operations + and
$\times$ are commutative and associative, and if $\times$ distributes over $+$. 
If there is an element $0 \in A$ such that $0+a = a +0 = a$ and $0\times a = a \times 0 = 0 $ for all $a\in A$, then $A$ is called a semiring with zero element. An element $1 \in A$ is said to be a unit element, if $1\times a = a \times 1 = a$ for all $a \in A$. In the sequel, we assume a semiring $A$ always has a unit element and a zero element. 

A semiring $A$ is said to have no multiplicative zero divisors if for any two non-zero elements $a, b\in A$, their product $a\times b\neq 0$. Equivalently, whenever $a \times b = 0$ for $a, b \in A$, it necessarily follows that $a=0$ or $b=0$.

Now we consider a non-empty finite set $r$ of variables with finite frames and a semiring $A$. We assume that at least one frame $\Omega_X$ contains at least two elements.
A semiring valuation $\phi$ with domain $d(\phi)=s\subseteq r$ is defined to be a function that associates a value from $A$ with a configuration ${\bf x}\in \Omega_s$, $\phi:\Omega_s \rightarrow A$. Let $\Phi$ denote the set of all valuations and $D={\cal P}(r)$.
We define the operations in the $(\Phi,D)$ using the operations $+$ and $\times$ in the semiring $A$:

(1) $Combination: \otimes :\Phi\times\Phi \rightarrow \Phi$, for $\phi,\psi\in \Phi$
and ${\bf x}\in \Omega_{d(\phi)\cup d(\psi)}$ we define
$$\phi\otimes\psi({\bf x})=\phi({\bf x}^{\downarrow {d(\phi)}})\times \psi({\bf x}^{\downarrow {d(\psi)}}).$$

(2)$Marginalization$: $\downarrow :\Phi\times D\rightarrow \Phi$ is defined for
all $\phi\in \Phi, t\subseteq d(\phi)=s$ and ${\bf x}\in \Omega_t$ by
$$\phi^{\downarrow t}({\bf x})=\sum\limits_{{\bf z}\in \Omega_s:{\bf z}^{\downarrow t}={\bf x}}
\phi({\bf z}).$$

It has been demonstrated that the pair $(\Phi,D)$ satisfies the axioms of valuation algebras, and is referred to as a valuation algebra induced by the semiring $A$ \cite{semiring}. We proceed to formally introduce this concept of valuation algebras in what follows.

\begin{definition}\label{definition:1}
	{\rm \textsuperscript{\cite{information}}}
	{\rm  Let $\Phi$ be a set of valuations with domains in $D$, the lattice of subsets of the set $r$ of variables. Suppose there are three operations defined:
		
		1.{\it Labeling} : $\Phi \rightarrow D ;\phi\mapsto d(\phi)$,
		
		2.{\it Combination} : $\Phi\times \Phi \rightarrow \Phi;(\phi,\psi)\mapsto \phi\otimes\psi$,
		
		3.{\it Projection} : $\Phi\times D \rightarrow \Phi;(\phi,x)\mapsto \phi^{\downarrow x}$.
		
		If the system $(\Phi,D)$ with these operations satisfying with the following axioms is called a valuation algebra:
		
		(1) {\it Semigroup}: $\Phi$ is associative and commutative under combination.
		For all $s\in D$ there is an element $e_s$ with $d(e_s)=s$ such that for all
		$\phi\in \Phi$ with $d(\phi)=s, e_s\otimes \phi=\phi\otimes e_s=\phi$.
		$e_s$ is called a neutral element of $\Phi_s$.
		
		(2) {\it Labeling}: For $\phi,\psi \in \Phi$, $d(\phi\otimes\psi) = d(\phi)\cup d(\psi)$.
		
		(3) {\it Marginalization}: For $\phi\in \Phi, x\in D, x\subseteq d(\phi), d(\phi^{\downarrow x})=x$.
		
		(4) {\it Transitivity}: For $\phi\in \Phi$ and $x\subseteq y\subseteq d(\phi),
		(\phi^{\downarrow y})^{\downarrow x}=\phi^{\downarrow x}$.
		
		(5) {\it Combination}: For $\phi,\psi \in \Phi$ with $d(\phi)=x, d(\psi)=y,
		(\phi\otimes\psi)^{\downarrow x}=\phi\otimes\psi^{\downarrow x\cap y}$.
		
		(6) {\it Neutrality}: 
		For $x,y\in D$, $e_x\otimes e_y=e_{x\cup y}$.
		
		Moreover, we often introduce a property in valuation algebras referred to as the {\it Stability} axiom:
		
		(7) {\it Stability} : For $x,y\in D, x\subseteq y$, $e_y^{\downarrow x}=e_x$.
		
	}
\end{definition}

The primary conclusions and examples on valuation algebras can be found in more details in \cite{information},\cite{semiring}.
In the following section, we investigate the relationship between isomorphisms of semirings and isomorphisms of their induced valuation algebras. 

\section{Semiring Isomorphism Induces Valuation Algebra Isomorphism}
\label{sec:2}

We first establish the easier direction: a semiring isomorphism induces an isomorphism between the corresponding semiring-induced valuation algebras. 

\begin{definition}\label{definition:4}
	{\rm \textsuperscript{\cite{Semiring}}}
	{ \rm A mapping $f$ from the semiring $A$ to the semiring $B$ is said to be a semiring homomorphism if for all $a, b \in A$, satisfying with 
		\begin{center}
			$f(a+b)= f(a)+ f(b),$\\ 
			$\hspace{0.1cm}f(a\times b)= f(a) \times f(b)$,\\
			$\hspace{-0.7cm}f(0)=0, f(1)=1$.
		\end{center}
		
		If $f$ is bijective, then $f$ is called a semiring isomorphism. 
}
\end{definition}

\begin{definition}\label{definition:5}
{\rm \textsuperscript{\cite{information}}}
{ \rm Let $(\Phi,D), (\Psi,D)$ be two valuation algebras, a mapping $h:(\Phi,D)\rightarrow (\Psi,D)$ is called a homomorphism, if, for $\phi_1,\phi_2, \phi \in \Phi,$
	
	1. $h(\phi_1\otimes \phi_2)=h(\phi_1)\otimes h(\phi_2)$;
	
	2. $h(\phi^{\downarrow x})=h(\phi)^{\downarrow x}$, if $ x\subseteq d(\phi)$;
	
	3. $h(e_x)=e_x$, for all $x\in D$.
	
	If $h$ is bijective, then $h$ is called an isomorphism. 
}
\end{definition}

\begin{theorem}\label{theorem:2}
{\rm Let $(\Phi, D), (\Psi,D)$ be two valuation algebras induced by two semirings $A$ and $B$ respectively. If $f:A\rightarrow B$ is a semiring isomorphism, then $(\Phi,D)$ is isomorphic to $(\Psi,D)$.}
\end{theorem}

\begin{proof}
For $\phi\in \Phi$, let $h(\phi)=f\circ \phi $. Then $h(\phi)\in\Psi$. We show that $h: \Phi \rightarrow \Psi$ is a homomorphism.

1. Let $\phi ,\psi \in \Phi $ with $d(\phi )=s, d(\psi )= t$.
Since $f$ is a semiring homomorphism, for all ${\bf x} \in \Omega _{s\cup t}$, we have
\begin{center}
	$h(\phi\otimes \psi)({\bf x})
	=f((\phi\otimes \psi) ({\bf x}))$\\
	$\hspace{2.6cm}= f(\phi ({\bf x}^{\downarrow s} ) \times \psi ({\bf x}^{ \downarrow t} ))$\\
	$\hspace{3.1cm}= f(\phi({\bf x}^{ \downarrow s} ))\times f(\psi ({\bf x}^{ \downarrow t} ))$\\
	$\hspace{3.1cm}=h(\phi)({\bf x}^{ \downarrow s})\times h(\psi)({\bf x}^{\downarrow t})$\\
	$\hspace{2.4cm}=(h(\phi) \otimes h(\psi))({\bf x})$. 
\end{center}
So $h(\phi \otimes\psi)= h(\phi) \otimes h(\psi)$.

2. Let $\phi \in \Phi $ with $d(\phi)= s$. If $t \subseteq s$, for all ${\bf x} \in \Omega _t$,
we have 
\begin{center}
	${h(\phi)}^{ \downarrow t} ( {\bf x})
	=\sum\limits_{ {\bf y} \in \Omega_{s-t}}{(f\circ \phi )( {\bf x}, {\bf y})}
	= f(\sum\limits_{ {\bf y} \in \Omega_{s-t} } {\phi ( {\bf x}, {\bf y})} )
	= f(\phi^{\downarrow t} ({\bf x}))=h(\phi^{\downarrow t})( {\bf x})$. 
\end{center}
Then ${h(\phi)}^{\downarrow t} = h(\phi^{\downarrow t})$.

3. In a semiring-induced valuation algebra, the neutral element $e_s$ of semigroup $\Phi_s$ is defined by $e_s({\bf x})=1$
for all ${\bf x}\in \Omega_s$. Because $f$ is a semiring homomorphism, we have $$h(e_s)({\bf x})=f(e_s({\bf x}))=f(1)=1, \forall {\bf x}\in \Omega_s.$$
Then $h(e_s)=e_s$.

We have shown that $h$ is a homomorphism. Next, it suffices to prove that $h$ is bijective.

If $h(\phi_1)= h(\phi_2)$, then 
$$d(\phi_1)=d(f\circ \phi_1)=d(h(\phi_1))=d(h(\phi_2))=d(f\circ \phi_2)=d(\phi_2).$$
Suppose that $d(\phi_1)=d(\phi_2)=s$. 
For all ${\bf x}\in \Omega_s$,
we have 
$$f(\phi_1({\bf x}))=(f\circ \phi_1)({\bf x})=
(f\circ \phi_2)({\bf x})=f(\phi_2({\bf x})).$$
Since $f$ is a bijection, it follows that $\phi_1({\bf x})=\phi_2({\bf x})$. The arbitrariness of ${\bf x}$ implies that $\phi_1=\phi_2$. Thus $h$ is injective.

For all $\psi: \Omega_s\rightarrow B$, we define a valuation $\phi: \Omega_s\rightarrow A$ by
$$\phi({\bf x})=f^{-1}(\psi({\bf x})).$$ 
Then $h(\phi)=\psi$. It means that $h$ is surjective. Therefore, $h$ is an isomorphism.
\end{proof}

\section{From Valuation Algebra Isomorphism to Semiring Isomorphism}

Suppose that $(\Phi,D)$ and $(\Psi,D)$ are two isomorphic valuation algebras induced by $A$ and $B$ respectively, with isomorphism $h$. Our goal is to extract a semiring isomorphism $f: A \rightarrow B$. 
In what follows, we assume that the semirings $A$ and $B$ are additively idempotent. It is well known that such semiring-induced valuation algebras satisfy the ${\it Stability}$ axiom(see \cite{semiring}). 

\begin{lemma}\label{lemma:16}
{\rm \textsuperscript{\cite{information}}
	If $h$ is a homomorphism between two valuation algebras $(\Phi,D)$ and $(\Psi,D)$ with the Stability, 
	then $h$ maintains the domain of a valuation, i.e., $$d(h(\phi))=d(\phi),\forall \phi\in \Phi.$$}
	\end{lemma}
	
	We begin by defining $f$ through the action of $h$ on valuations with empty domain. Since homomorphisms between stable valuation algebras maintain domains of valuations, we may define a mapping $f: A \longrightarrow B$ by: 
	$$\forall a\in A, f(a)=h(\zeta^{(a)})(\epsilon),$$ 
	where the valuation $\zeta^{(a)}: \Omega_{\emptyset} \longrightarrow A$ is given by $\zeta^{(a)}(\epsilon)=a$.
	
	\subsection{Multiplication Preservation via Valuation Algebra Isomorphisms}
	In this subsection, for $f: A \rightarrow B$ between the underlying semirings, we establish its basic algebraic properties -- namely bijectivity, preservation of the zero and unit elements, and preservation of multiplication.
	
	\begin{lemma}\label{lemma:6}
{\rm	The mapping $f$ is a bijection.}
\end{lemma}
\begin{proof}
Let $a_1, a_2\in A$ and suppose that 
$f(a_1) = f(a_2)$, i.e., $h(\zeta^{(a_1)})(\epsilon) = h(\zeta^{(a_2)})(\epsilon)$.  
Since $h(\zeta^{(a_1)})$ and $h(\zeta^{(a_2)})$ are valuations in $\Psi_\emptyset$ and $\Omega_\emptyset$ is a singleton, it follows that
$h(\zeta^{(a_1)}) = h(\zeta^{(a_2)})$.  
As $h$ is injective, we obtain $\zeta^{(a_1)} = \zeta^{(a_2)}$, which implies $a_1 = a_2$. Thus $f$ is injective.

Since $h$ is surjective, for any $b \in B$, 
there exists $\zeta^{(a)}\in \Phi_\emptyset$ such that $h(\zeta^{(a)}) = \zeta^{(b)}$.  
It follows that $f(a) = h(\zeta^{(a)})(\epsilon)=\zeta^{(b)}(\epsilon) = b$. Hence,  $f$ is surjective.
\end{proof}

\begin{lemma}\label{lemma:13}
{\rm The mapping $f$ preserves the zero element and the unit element, i.e., $f(0)= 0$ and $f(1)= 1.$}
\end{lemma}

\begin{proof}
Since $ f $ is surjective, for any $ b \in B $ there exists $ a \in A $ such that $ f(a) = b $.
Then 
\begin{center}
	$f(0) \cdot b=f(0) \cdot f(a) $\\
	$\hspace{2.6cm}= h(\zeta^{(0)})(\epsilon) \cdot h(\zeta^{(a)})(\epsilon)$\\
	$\hspace{2.8cm}=\left[h(\zeta^{(0)}) \otimes h(\zeta^{(a)})\right](\epsilon)$\\
	$\hspace{2cm}=h(\zeta^{(0)} \otimes \zeta^{(a)})(\epsilon)$\\
	$\hspace{1.1cm}=h(\zeta^{(0)})(\epsilon)$\\
	$\hspace{0.5cm}=f(0) .$
\end{center}
In particular, taking $ b = 0$(in the semiring $B$), we get $ f(0) = f(0) \cdot 0=0$.

Since the valuation homomorphism $h$ maps the neutral element $\zeta^{(1)}$ of $\Phi_\emptyset$ to the neutral element $\zeta^{(1)}$ of $\Psi_\emptyset$, then  

{\centering \hspace{4cm}
	$f(1) = h(\zeta^{(1)})(\epsilon) 
	= \zeta^{(1)}(\epsilon) = 1.$}
	\end{proof}
	
	\begin{lemma}\label{lemma:3}
{\rm The mapping $f$ preserves the multiplication operation.}
\end{lemma}
\begin{proof}
For all $a_1, a_2 \in A$, we have
\begin{center}
	$f(a_1 \cdot a_2) 
	= h(\zeta^{(a_1 \cdot a_2)})(\epsilon)$\\
	$\hspace{2.2cm}=h(\zeta^{(a_1)} \otimes \zeta^{(a_2)})(\epsilon)$\\
	$\hspace{2.8cm}=  [h(\zeta^{(a_1)}) \otimes h(\zeta^{(a_2)})](\epsilon)$\\
	$\hspace{2.8cm}=h(\zeta^{(a_1)})(\epsilon) \cdot h(\zeta^{(a_2)})(\epsilon)$\\
	$\hspace{1.5cm}=f(a_1) \cdot f(a_2).$
\end{center}
Then the mapping $f$ preserves the multiplication.
\end{proof}

Under the assumption that the semirings are additively idempotent, we prove that the mapping $f$ satisfies all conditions required for an isomorphism except preserving the addition operation.

\subsection{Properties of Elementary Valuations and Their Images}

To proceed further, we need to understand how the homomorphism $h: \Phi \rightarrow \Psi$ behaves on elementary valuations that are concentrated at a single configuration. These valuations play a fundamental role in a valuation algebra, similar to the basis elements in linear spaces.

Let $s\subseteq r$, $\mathbf{x} \in \Omega_s$ and $a \in A$. Define a valuation $\delta_{\mathbf{x}}^{(a)}: \Omega_s \to A$ by:
$$\delta_{\mathbf{x}}^{(a)}(\mathbf{x}') =\left\{
\begin {array}{ll}
a,&{\mbox{if}} \ \ \mathbf{x}' = \mathbf{x};\\
0,&{\mbox{otherwise}}.
\end{array}
\right.$$
We write $\delta_{\mathbf{x}}^{(1)}$ as $\delta_{\mathbf{x}}$, or $\delta_{\mathbf{x}}^A$ when it is necessary to distinguish the underlying semirings. There exists a connection between mappings of the form $\zeta^{(a)}$ and $\delta_{\mathbf{x}}^{(a) }$. We present some properties of these special valuations.

\begin{lemma}\label{lemma:1}
$\left[\delta_{\mathbf{x}}^{(a) }\right]^{\downarrow \emptyset} =\zeta^{(a)}$.
\end{lemma}

\begin{proof}
Since
\begin{center} 
$\left[\delta_{\mathbf{x}}^{(a) }\right]^{\downarrow \emptyset}(\epsilon) = \sum\limits_{\mathbf{x}' \in \Omega_s} \delta_{\mathbf{x}}^{(a)}(\mathbf{x}') = \delta_{\mathbf{x}}^{(a)}(\mathbf{x}) + \sum\limits_{\mathbf{x}' \neq \mathbf{x}} \delta_{\mathbf{x}}^{(a)}(\mathbf{x}') = a $,
\end{center}
then $\left[\delta_{\mathbf{x}}^{(a) }\right]^{\downarrow \emptyset} =\zeta^{(a)}$.
\end{proof}

\begin{lemma}\label{lemma:5}
$\zeta^{(a)}\otimes \delta_{\mathbf{x}}=\delta_{\mathbf{x}}^{(a)}$.
\end{lemma}	
\begin{proof}
For $\mathbf{x}'\in \Omega_s$,
\begin{center}
$ [\zeta^{(a)}\otimes \delta_{\mathbf{x}}] (\mathbf{x}')
=  \zeta^{(a)}(\epsilon) \times \delta_{\mathbf{x}} (\mathbf{x}') 
=a \times \delta_{\mathbf{x}} (\mathbf{x}').$
\end{center}
If $ \mathbf{x}' = \mathbf{x} $, then 
$$[\zeta^{(a)}\otimes \delta_{\mathbf{x}}] (\mathbf{x}) = a = \delta_{\mathbf{x}}^{(a)}(\mathbf{x}).$$
If $ \mathbf{x}' \neq \mathbf{x} $, then 
$$[\zeta^{(a)}\otimes \delta_{\mathbf{x}}] (\mathbf{x}')
= 0= \delta_{\mathbf{x}}^{(a)}(\mathbf{x}').$$
This means 
$[\zeta^{(a)}\otimes \delta_{\mathbf{x}}] (\mathbf{x}')
=\delta_{\mathbf{x}}^{(a)}(\mathbf{x}')$ for all $\mathbf{x}'\in \Omega_s$.
Thus $\zeta^{(a)}\otimes \delta_{\mathbf{x}}=\delta_{\mathbf{x}}^{(a)}$.
\end{proof}

\begin{lemma}\label{lemma:8}
{\rm	Let \(h: (\Phi,D) \to (\Psi,D)\) be a valuation algebra isomorphism induced by semirings. Then \(h(\mathbf{0}_s) = \mathbf{0}_s\) holds for all \(s\subseteq r\), where \(\mathbf{0}_s\) denotes the valuation that constantly takes the value \(0\) with domain $s$.
}
\end{lemma}

\begin{proof}
The zero valuation \(\mathbf{0}_s\) satisfies that for any \(\phi \in \Phi_s\),
\[
\phi \otimes \mathbf{0}_s = \mathbf{0}_s.
\]
It then follows that
$$
h(\mathbf{0}_s)= h(\phi \otimes \mathbf{0}_s) =h(\phi) \otimes h(\mathbf{0}_s) .
$$
Since \(h\) is surjective and label-preserving, there exists a valuation \(\phi \in \Phi_s\) such that \(h(\phi)=\mathbf{0}_s\). 
This implies that 

{\centering \hspace{5cm}
$
h(\mathbf{0}_s)=\mathbf{0}_s \otimes h(\mathbf{0}_s)=\mathbf{0}_s.
$}
\end{proof}

\begin{lemma}\label{lemma:7}
{\rm Let \(h: (\Phi,D) \to (\Psi,D)\) be a valuation algebra isomorphism induced by semirings. If $\mathbf{x},\mathbf{y}\in \Omega_s$ and $\mathbf{x}\neq \mathbf{y}$, 
then $h(\delta_{\mathbf{x}})(\mathbf{z})\times h(\delta_{\mathbf{y}})(\mathbf{z})=0$ for all $\mathbf{z}\in \Omega_s$.}
\end{lemma}
\begin{proof} By Lemma \ref{lemma:8}, we have
\[
h(\delta_{\mathbf{x}})(\mathbf{z})\times h(\delta_{\mathbf{y}})(\mathbf{z})=[h(\delta_{\mathbf{x}})\otimes h(\delta_{\mathbf{y}})](\mathbf{z})=h(\delta_{\mathbf{x}}\otimes \delta_{\mathbf{y}})(\mathbf{z})=h({\mathbf{0}_s})(\mathbf{z})=0.
\]
\end{proof}

\begin{proposition}\label{pro:3}
{\rm Let \(h: (\Phi,D) \to (\Psi,D)\) be a valuation algebra isomorphism induced by additive idempotent semirings. $B$ is a semiring with no multiplicative idempotent zero divisors. If $ s\subseteq r$ and $ \mathbf{x} \in \Omega_s $, then there exists a unique $ {\bf y} \in \Omega_s $ such that
$ h(\delta_\mathbf{x})({\bf y}) = 1_B $ 
and $\forall \mathbf{z}\in \Omega_s, \mathbf{z}\neq {\bf y}, \ h(\delta_\mathbf{x})(\mathbf{z}) = 0_B. $}
\end{proposition}

\begin{proof}
For $\mathbf{z} \in \Omega_s $, 
it is known that
$ (\delta_\mathbf{x} \otimes \delta_\mathbf{x})(\mathbf{z}) 
= \delta_\mathbf{x}(\mathbf{z}) \times \delta_\mathbf{x}(\mathbf{z}). $
If $ \mathbf{z} = \mathbf{x} $, then $ (\delta_\mathbf{x} \otimes \delta_\mathbf{x})(\mathbf{x}) = 1_A \times 1_A = 1_A = \delta_\mathbf{x}(\mathbf{x})$.
When $ \mathbf{z} \neq \mathbf{x} $,  $(\delta_\mathbf{x} \otimes \delta_\mathbf{x})(\mathbf{z})=0_A \times 0_A = 0_A = \delta_\mathbf{x}(\mathbf{z}) $.
Thus $ \delta_\mathbf{x} \otimes \delta_\mathbf{x} = \delta_\mathbf{x}$. That is, $ \delta_\mathbf{x}$ is  idempotent under combination.

Since $ h $ is a homomorphism, we have
$
h(\delta_\mathbf{x}) \otimes h(\delta_\mathbf{x}) = h(\delta_\mathbf{x} \otimes \delta_\mathbf{x}) = h(\delta_\mathbf{x}).
$
Hence, for any $ \mathbf{z} \in \Omega_s $, 
\begin{center}
$h(\delta_\mathbf{x})(\mathbf{z})=
[h(\delta_\mathbf{x}) \otimes h(\delta_\mathbf{x})](\mathbf{z}) = h(\delta_\mathbf{x})(\mathbf{z}) \times h(\delta_\mathbf{x})(\mathbf{z}).
$
\end{center}
This shows that each $h(\delta_\mathbf{x})(\mathbf{z})$ is a multiplicative idempotent element in the semiring $B$. 

Secondly, the sum of all output value of $ h(\delta_\mathbf{x}) $ equals $ 1_B $.
By Lemma \ref{lemma:1} and the fact that $h$ preserves marginalization,
\begin{center}
$ h\left(\zeta^{(1_A)}\right)
=h\left[\left(\delta_{\mathbf{x}}\right)^{\downarrow \emptyset} \right]
= \left[ h(\delta_{\mathbf{x}}) \right]^{\downarrow \emptyset}$.
\end{center}	
Thus
\begin{center} 
$1_B= f(1_A)
= h\left(\zeta^{(1_A)}\right) (\epsilon) 
= \left[ h(\delta_{\mathbf{x}}) \right]^{\downarrow \emptyset}(\epsilon)
=\sum\limits_{\mathbf{z} \in \Omega_s}  h(\delta_{\mathbf{x}})(\mathbf{z})$.
\end{center}	

Now suppose there exists an ${\mathbf{x}}_1$ such that $h(\delta_{{\mathbf{x}}_1})$ takes nonzero values at two distinct elements ${\mathbf{x}}_2, {\mathbf{x}}_3$. By Lemma \ref{lemma:7}, the remaining $|\Omega_s|-1$ functions $h(\delta_{\mathbf{x}})$ with ${\mathbf{x}} \neq {\mathbf{x}}_1$ can only take nonzero values at configurations other than ${\mathbf{x}}_2, {\mathbf{x}}_3$. The number of remaining available configurations is $|\Omega_s| - 2$. Consequently, at least two functions take nonzero values simultaneously at a same configuration, such as $h(\delta_{\mathbf{x}_4})(\mathbf{x}_6)\neq 0$ and 
$h(\delta_{\mathbf{x}_5})(\mathbf{x}_6)\neq 0$. If $B$ is a semiring with no multiplicative idempotent zero divisors, it yields $h(\delta_{\mathbf{x}_4})(\mathbf{x}_6)\times 
h(\delta_{\mathbf{x}_5})(\mathbf{x}_6)\neq 0$. It is a contradiction with Lemma \ref{lemma:7}.

Therefore, each $h(\delta_{\mathbf{x}})$ must take a nonzero value at exactly one configuration, and by the sum condition $1_B$, this value must be exactly $1_B$. That is, there exists a unique $\mathbf{y} \in \Omega_s$ such that
$h(\delta_\mathbf{x})(\mathbf{y}) = 1_B, $
and for all $ \mathbf{z}\neq \mathbf{y} $,
$h(\delta_\mathbf{x})(\mathbf{z}) = 0_B $.
\end{proof}

Proposition \ref{pro:3} shows that for each $\mathbf{x}\in \Omega_s$, there exists a unique $\mathbf{y}\in \Omega_s$ such that $h(\delta_\mathbf{x}^A)=\delta_\mathbf{y}^B$. The same conclusion holds for the isomorphism $h^{-1}$. Then, for each single variable $X \in r$, we can define a permutation $\sigma_X: \Omega_X \to \Omega_X$ such that for any $\omega \in \Omega_X$,
$h(\delta_\omega^{A}) = \delta_{\sigma_X(\omega)}^{B} $.

For any domain $s \subseteq r$, define a mapping $\sigma_s: \Omega_s \to \Omega_s$ by:
$$\sigma_s\left((\omega_X)_{X \in s}\right) 
= (\sigma_X(\omega_X))_{X \in s}.$$
That is, $\sigma_s$ acts componentwise.   

\begin{lemma}\label{lemma:12}
{\rm If $t \subseteq s\subseteq r$ and $\mathbf{x}\in \Omega_s$, then
$\left[\sigma_s(\mathbf{x})\right]^{\downarrow t} = \sigma_t(\mathbf{x}^{\downarrow t})$. }
\end{lemma}

\begin{proof} 
Let  $\mathbf{x} = (\omega_X)_{X \in s} \in \Omega_s$. According to the definition of $\tau_s$,  
for $t \subseteq s$, the marginalization $\left[\sigma_s(\mathbf{x})\right]^{\downarrow t}$ keeps the components corresponding to the variables in $t$, i.e.,  
$$\left[\sigma_s(\mathbf{x})\right]^{\downarrow t}= 
\big( \sigma_X(\omega_X) \big)_{X \in t}.$$

On the other hand, $\mathbf{x}^{\downarrow t}= (\omega_X)_{X \in t}$, and hence  
$$\sigma_t(\mathbf{x}^{\downarrow t}) = \big( \sigma_X(\omega_X) \big)_{X \in t}.$$
Therefore, $\left[\sigma_s(\mathbf{x})\right]^{\downarrow t}= \sigma_t(\mathbf{x}^{\downarrow t})$ holds for each $\mathbf{x} \in \Omega_s$.
\end{proof}

We refer to this as the commutativity of permutation and projection, or Property C for short. The mapping $\sigma_s^{-1}$ clearly possesses this Property C as well.

Combined with the mapping $\sigma_s$, for any $\psi \in \Psi_s$, define a valuation $T(\psi): \Omega_s \to  B$ by:
$$
T(\psi)(\mathbf{x}) = \psi\big( \sigma_s(\mathbf{x}) \big), \quad \forall \mathbf{x} \in \Omega_s.
$$
In other words, $T(\psi) = \psi \circ \sigma_s$. 
Clearly $d(T(\psi)) = d(\psi)$. 
Thus, using the mapping $T$, we construct a new valuation algebra $\Psi' = \{T(\psi) : \psi \in \Psi \}$.

\begin{proposition}\label{pro:4}
{\rm $T: \Psi \to \Psi'$ is an isomorphism of valuation algebras.}
\end{proposition}

\begin{proof}
1. Let $\psi_1, \psi_2 \in \Psi$ with domains $s_1, s_2$ respectively.  
For any $\mathbf{x} \in \Omega_{s_1 \cup s_2}$,  
$$
T(\psi_1 \otimes \psi_2)(\mathbf{x})
= (\psi_1 \otimes \psi_2)(\sigma_{s_1\cup s_2}(\mathbf{x})) 
= \psi_1\big( (\sigma_{s_1\cup s_2}(\mathbf{x}))^{\downarrow s_1} \big) 
\times 
\psi_2\big( (\sigma_{s_1\cup s_2}(\mathbf{x}))^{\downarrow s_2} \big).
$$  
By Property C, $(\sigma_{s_1\cup s_2}(\mathbf{x}))^{\downarrow s_i} = \sigma_{s_i}(\mathbf{x}^{\downarrow s_i})$,  
so  
$$ T(\psi_1 \otimes \psi_2)(\mathbf{x})
= \psi_1(\sigma_{s_1}(\mathbf{x}^{\downarrow s_1})) \times \psi_2(\sigma_{s_2}(\mathbf{x}^{\downarrow s_2}))
= T(\psi_1)(\mathbf{x}^{\downarrow s_1}) \times T(\psi_2)(\mathbf{x}^{\downarrow s_2}).
$$  
It is exactly $(T(\psi_1) \otimes T(\psi_2))(\mathbf{x})$.  Hence $T$ preserves combination of valuations.

2. For $\psi \in \Psi_s$, $t \subseteq s$, and any $\mathbf{y} \in \Omega_t$,  
\begin{center}
$[T(\psi)^{\downarrow t}](\mathbf{y})
= \sum\limits_{{\mathbf{x} \in \Omega_s, \mathbf{x}^{\downarrow t} = \mathbf{y}}} T(\psi)(\mathbf{x}) 
= \sum\limits_{{\mathbf{x} \in \Omega_s, \mathbf{x}^{\downarrow t} = \mathbf{y}}} \psi(\sigma_s(\mathbf{x})).
$
\end{center}  
Let $\mathbf{x}' = \sigma_s(\mathbf{x})$. Then $\mathbf{x}'^{\downarrow t} = \sigma_t(\mathbf{y})$ (by Property C).  
There is a one-to-one correspondence between the elements of set $\{\mathbf{x} \in \Omega_s: \mathbf{x}^{\downarrow t} = \mathbf{y}\}$ and set $\{\mathbf{x}' \in \Omega_s: \mathbf{x}'^{\downarrow t}	= \sigma_t(\mathbf{y})\}$.
So  
\begin{center}
$[T(\psi)^{\downarrow t}](\mathbf{y})
= \sum\limits_{{\mathbf{x}' \in \Omega_s, \mathbf{x}'^{\downarrow t} 
		= \sigma_t(\mathbf{y})}} \psi(\mathbf{x}')
= \psi^{\downarrow t}(\sigma_t(\mathbf{y}))
= [T(\psi^{\downarrow t})](\mathbf{y}).
$
\end{center}  
Thus $\left[T(\psi)\right]^{\downarrow t} = T(\psi^{\downarrow t})$.

3. For any domain $s \subseteq r$, let $e_s$ be the mapping on domain $s$ with constant value $1$. 	
By the definition of $T$, for any $\mathbf{x} \in \Omega_s$,  
$ T(e_s)(\mathbf{x}) 
= e_s\big( \sigma_s(\mathbf{x}) \big)= 1.$  
This means $T(e_s)$ is the neutral element on domain $s$ in  $\Psi'$.

4. Let $\psi_1, \psi_2 \in \Psi_s$ with $T(\psi_1)=T(\psi_2)$.
For all $\mathbf{x} \in \Omega_s$,  
let $\mathbf{y}=\sigma^{-1}_s(\mathbf{x})$. Then 
$T(\psi_1)(\mathbf{y})=T(\psi_2)(\mathbf{y})$, that is,
$\psi_1(\sigma_s(\mathbf{y}))=\psi_2(\sigma_s(\mathbf{y}))$. It follows that 
$\psi_1(\mathbf{x})= \psi_2(\mathbf{x})$.
As this holds for all $\mathbf{x} \in \Omega_s$, we conclude that $\psi_1=\psi_2$. Thus $T$ is injective.  

Therefore $T: \Psi \to \Psi'$ is an isomorphism of valuation algebras. 
\end{proof}

Next we show below that it is possible to construct a new isomorphism $h'$ between valuation algebras $(\Phi, D)$ and $(\Psi', D)$.

\begin{proposition}\label{pro:5}
{\rm Let $h: \Phi \to \Psi$ be a valuation algebra isomorphism, and suppose that $A$ and $B$ are two additive  idempotent semirings with no multiplicative idempotent zero divisors. Then there exists a valuation algebra isomorphism $h': \Phi \to \Psi'$ such that for any domain $s \subseteq r$ and any $\mathbf{z} \in \Omega_s$,
$ h'(\delta_{\mathbf{z}}^{A}) = \delta_{\mathbf{z}}^{B}.$}
\end{proposition} 

\begin{proof}
Define a new isomorphism $h' = T \circ h$. Clearly, 
$h': \Phi \to \Psi'$ is a composition of isomorphisms.
For any domain $s$ and any $\mathbf{z} \in \Omega_s$,  
we have
$$ h'(\delta_{\mathbf{z}}^{A}) 
= T \circ h(\delta_{\mathbf{z}}^{A})
= T\big( \delta_{\sigma_s(\mathbf{z})}^{B} \big).$$
For any $\mathbf{x} \in \Omega_s$,
$$  h'(\delta_{\mathbf{z}}^{A}) (\mathbf{x})=
T(\delta_{\sigma_s(\mathbf{z})}^{B})(\mathbf{x}) = \delta_{\sigma_s(\mathbf{z})}^{B}(\sigma_s(\mathbf{x})).
$$
The above result equals $1_B$ if and only if $\sigma_s(\mathbf{x}) = \sigma_s(\mathbf{z})$, i.e., $\mathbf{x}= \mathbf{z}$. 
Therefore
\begin{center}
$ h'(\delta_{\mathbf{z}}^{A})(\mathbf{x})
= \delta_{\mathbf{z}}^{B}(\mathbf{x})$.
\end{center}
Hence $ h'(\delta_{\mathbf{z}}^{A}) =\delta_{\mathbf{z}}^{B}$
by the arbitrariness of $\mathbf{x}$.
\end{proof}

The structure of a semiring-induecd valuation algebra is completely determined by the semiring $B$ and the variable set $r$ (together with theirs configurations).  
Hence $(\Psi', D)$ and $(\Psi, D)$ are essentially identical valuation algebras. 
In the following we may assume that 
$h(\delta_\mathbf{z}^A)(\mathbf{z}) = 1_B$ under the isomorphism $h$.

\begin{lemma}\label{lemma:2}
{\rm If $\mathbf{x}\in \Omega_s$ and $a\in A$, then
$h(\delta_{\mathbf{x}}^{(a)})(\mathbf{x})= f(a)$.}
\end{lemma}	
\begin{proof}
By Lemma \ref{lemma:7}, we have $h(\delta_{\mathbf{x}}^A)(\mathbf{x})= 1_B$.
Then, by Lemma \ref{lemma:5}, 
\begin{center}
$h(\delta_{\mathbf{x}}^{(a)})(\mathbf{x})
= h(\zeta^{(a)}\otimes \delta_{\mathbf{x}}^A)(\mathbf{x})$\\
$\hspace{2.4cm}=\left[h(\zeta^{(a)})\otimes h(\delta_{\mathbf{x}}^A)\right](\mathbf{x})$\\
$\hspace{2.5cm}= h(\zeta^{(a)})(\epsilon)\times h(\delta_{\mathbf{x}}^A)(\mathbf{x})$\\
$\hspace{1.8cm}=f(a)\times h(\delta_{\mathbf{x}}^A)(\mathbf{x})$\\
$\hspace{0.3cm}=f(a)$.
\end{center}
\end{proof}


\subsection{Preservation of Addition and the Main Theorem}

With the behavior of $h$ on elementary valuations fully characterized, we are now in a position to prove that the induced mapping $f$ preserves addition, which is the remaining property required for a semiring isomorphism.

\begin{lemma}\label{lemma:4}
{\rm Let $(\Phi, D)$ and $(\Psi, D)$ be two valuation algebras  induced by the additive idempotent semirings $A$ and $B$ with no multiplicative idempotent zero divisors, respectively. Suppose that $h: \Phi \to \Psi$ is an isomorphism. For any valuation $\theta \in \Phi_s$ and any $\mathbf{x} \in \Omega_s$, we have
$$ h(\theta)(\mathbf{x}) = f( \theta(\mathbf{x})).$$ }
\end{lemma}

\begin{proof}
Let $a\in A$, for the combination of $\theta$ and $\delta_\mathbf{x}^{(a)}$:
$$(\theta \otimes \delta_\mathbf{x}^{(a)})(\mathbf{x}') = \theta(\mathbf{x}') \cdot \delta_\mathbf{x}^{(a)}(\mathbf{x}'), 
\forall \mathbf{x}' \in \Omega_s.$$
It is nonzero only at the configuration $\mathbf{x}$:
$(\theta \otimes \delta_\mathbf{x}^{(a)})(\mathbf{x}) = \theta(\mathbf{x}) \cdot a$. 
Thus $\theta \otimes \delta_\mathbf{x}^{(a)}= \delta_\mathbf{x}^{(\theta(\mathbf{x}) \cdot a)}$.
By Lemma \ref{lemma:2} and Lemma \ref{lemma:3}, we have
$$ h(\theta \otimes \delta_\mathbf{x}^{(a)}) (\mathbf{x})
=h(\delta_\mathbf{x}^{(\theta(\mathbf{x}) \cdot a)})(\mathbf{x})
=f \left( \theta(\mathbf{x}) \cdot a \right)
=f(\theta(\mathbf{x})) \cdot f(a).$$

On the other hand, since the isomorphism $h$ preserves the combination operation, we have
$h(\theta \otimes \delta_\mathbf{x}^{(a)}) = h(\theta) \otimes h(\delta_\mathbf{x}^{(a)})$. 
Therefore
$$ h(\theta \otimes \delta_\mathbf{x}^{(a)}) (\mathbf{x})= \left[ h(\theta) \otimes h(\delta_\mathbf{x}^{(a)}) \right] (\mathbf{x})= h(\theta)(\mathbf{x}) \cdot  h(\delta_\mathbf{x}^{(a)})(\mathbf{x}).$$ 
By Lemma \ref{lemma:2}, we obtain that
$ h(\theta \otimes \delta_\mathbf{x}^{(a)})(\mathbf{x}) = \  h(\theta)(\mathbf{x}) \cdot f(a)$.

Since $a \in A$ is arbitrary, we can take $a=1_A$, it follows that

{\centering \hspace{5cm}
$h(\theta)(\mathbf{x}) =f(\theta(\mathbf{x})).$
}
\end{proof}

\begin{lemma}\label{lemma:11}
{\rm The mapping $f$ preserves the addition operation.}
\end{lemma} 
\begin{proof}
Let $\omega_1 \neq \omega_2$ be two distinct values of variable
variable $X$ in $\Omega_X$. For any $a, b \in A$, define a valuation $\gamma: \Omega_X \to A$:
$$\gamma(\mathbf{x}) =\left\{
\begin {array}{ll}
a,&{\mbox{if}} \ \ \mathbf{x}= \omega_1;\\
b,&{\mbox{if}} \ \ \mathbf{x}= \omega_2;\\
0_A,&{\mbox{otherwise}}.
\end{array}
\right.$$
Then $\gamma^{\downarrow\emptyset}=\zeta^{(a+b)}$.
So, by Lemma \ref{lemma:4},
$$h(\gamma^{\downarrow\emptyset})(\epsilon)=
f\left(\gamma^{\downarrow\emptyset}(\epsilon)\right)
=f\left(\zeta^{(a+b)}(\epsilon)\right)=f(a+b).$$

On the other hand, by Lemma \ref{lemma:4}, we also obtain that  $$\left[h(\gamma)\right]^{\downarrow\emptyset}(\epsilon) = \sum_{\mathbf{x}\in \Omega_X} h(\gamma)(\mathbf{x})= \sum_{\mathbf{x}\in \Omega_X} f(\gamma(\mathbf{x}))=f(a)+f(b).$$	
Since $h$ is a homomorphism, we have $$h(\gamma^{\downarrow\emptyset})(\epsilon)
=\left[h(\gamma)\right]^{\downarrow\emptyset}(\epsilon), i.e.,
f(a+b)=f(a)+f(b).$$
It has shown that the mapping $f$ preserves the addition.
\end{proof}

By Lemma \ref{lemma:6}, Lemma \ref{lemma:13}, Lemma \ref{lemma:3} and Lemma \ref{lemma:11}, we obtain the conclusion as follows.

\begin{theorem}\label{Th:2}
{\rm Let $(\Phi,D)$ and $(\Psi,D)$ be two isomorphic valuation algebras induced by the semirings $A$ and $B$, respectively, where $A$ and $B$ are additively idempotent and have no multiplicative idempotent zero divisors. Then $f: A \to B$ is an isomorphism.}
\end{theorem}

Through the above discussion, we have established the connection between semirings and the valuation algebras induced by these semirings in terms of algebraic structures.

\section{Conclusion}
The main contribution of this paper is Theorem 1 and Theorem 2, which establishes that, for additively idempotent semirings without multiplicatively idempotent zero divisors, the existence of an isomorphism between the valuation algebras induced by two semirings is equivalent to the existence of an isomorphism between the semirings themselves.

This result complements and extends the existing theory of valuation algebra homomorphisms initiated by Kohlas [3]. In his foundational work, Kohlas introduced the notion of valuation algebra homomorphisms and established several fundamental properties. However, the specific problem of characterizing when an isomorphism between semiring-induced valuation algebras descends to an isomorphism of the underlying semirings was not addressed. Our result is to provide a complete structural correspondence for the important subclass of semiring-induced valuation algebras. Thus, our work bridges the gap between the abstract theory of valuation algebras and the concrete algebraic structure of their semiring generators.

\nocite{*}
\bibliographystyle{fundam}
\bibliography{citations}


\end{document}